\documentclass[a4paper]{amsart}

\usepackage[T1]{fontenc}
\usepackage{lmodern}

\usepackage{
    amssymb,
    booktabs,
}

\usepackage{imakeidx}

\usepackage{hyperref}
\hypersetup{colorlinks, allcolors=blue}
\usepackage[capitalise]{cleveref}
\crefdefaultlabelformat{#2\textup{#1}#3} 

\usepackage[backend=biber,giveninits=true,uniquename=false,doi=true]{biblatex}
\renewbibmacro{in:}{}
\DeclareFieldFormat{pages}{#1}

\usepackage[pagewise]{lineno}

\usepackage{orcidlink}

\theoremstyle{plain}
\newtheorem{mainthm} {Theorem}

\newtheorem{maincor} [mainthm]{Corollary}
\newtheorem{thm} {Theorem} [section]

\newtheorem{lem}    [thm] {Lemma}
\theoremstyle{definition}

\newtheorem{rmk}    [thm] {Remark}

\crefname{main}{Theorem}{Theorems}
\crefname{thm}{Theorem}{Theorems}
\crefname{lem}{Lemma}{Lemmas}
\crefname{prop}{Proposition}{Propositions}
\crefname{cor}{Corollary}{Corollaries}
\crefname{rmk}{Remark}{Remarks}
\crefname{tbl}{Table}{Tables}

\numberwithin{equation}{section}

\DeclareMathOperator{\GL}{GL}
\DeclareMathOperator{\Irr}{Irr}
\DeclareMathOperator{\IBr}{IBr}
\DeclareMathOperator{\I}{I}
\DeclareMathOperator{\cl}{cl}
\DeclareMathOperator{\rank}{rank}
\DeclareMathOperator{\diag}{diag}
\renewcommand{\phi}{\varphi}
\renewcommand{\setminus}{\smallsetminus}
\newcommand{\pp}{\mathfrak{p}}
\newcommand{\OO}{\mathcal{O}}
\newcommand{\C}{\mathbf{C}}
\newcommand{\QQ}{\mathbb{Q}}
\newcommand{\ZZ}{\mathbb{Z}}

\title[Elementary divisors of the Cartan matrix]{Elementary divisors of the Cartan matrix for\\ partial characters}
\author[K.~Dastouri]{Kaveh Dastouri\,\orcidlink{0000-0001-8075-5515}}
\address[Kaveh Dastouri]
{Beijing International Center for Mathematical Research,
Peking University, Beijing, 100871, China}
\email{k.dastouri@bicmr.pku.edu.cn}

\author[T.~Sakurai]{Taro Sakurai\,\orcidlink{0000-0003-0608-1852}}
\address[Taro Sakurai]
{Department of Mathematics and Informatics, Graduate School of Science, Chiba University, 1-33 Yayoi-cho, Inage-ku, Chiba-shi, Chiba, 263-8522, Japan.}
\email{tsakurai@math.s.chiba-u.ac.jp}

\date{\today}

\begin{document}

\begin{abstract}
  Let $\pi$ be a set of primes and let $G$ be a finite $\pi$-separable group.
  We prove that the Cartan matrix $C$ for the $\pi$-partial characters of $G$ is equivalent over the integers to a matrix $\operatorname{diag}(|\mathbf{C}_G(x)|_{\pi'})$, where $x$ runs over a set of representatives of the $\pi$-classes of $G$.
  In particular, we prove that $\det C = \prod |\mathbf{C}_G(x)|_{\pi'}.$
\end{abstract}
\index{$\pi$ : Set}
\index{$\pi'$ : Set}
\index{$G$ : Group}
\index{$C$ : Matrix}
\index{$x$ : Group Element}
\index{$C_G(x)$ @ $\C_G(x)$ : Group}
\maketitle


\section{Introduction}
Brauer proved in \cite[Theorem~1]{Brauer41} that the determinant of the Cartan matrix $C$ for the $p$-Brauer characters of a finite group $G$ is a power of $p$.
In fact, Brauer and Nesbitt determined the elementary divisors\footnote{We follow the terminology of \cite[p.~68]{NagaoTsushima89}. These are also called invariant factors in \cite{Bourbaki90}, for example.} of the Cartan matrix $C$ in \cite[\S16]{BrauerNesbitt41}.
These are the orders of the Sylow $p$-subgroups of the centralizers $\C_G(x)$, where $x$ runs over a set of representatives of the $p'$-classes of $G$.
This result reveals a connection between the structure of projective modules and Sylow $p$-subgroups.
Modern accounts appear in \cite[Theorem~3.6.32]{NagaoTsushima89} and \cite[Exercise~18.5(b)]{Serre77}.

Isaacs \cite{Isaacs84} began to generalize the $p$-Brauer characters of finite $p$-solvable groups by replacing $p'$-classes with $\pi$-classes, where $\pi$ is a set of primes.
The Cartan matrix is also defined for the $\pi$-partial characters of a finite $\pi$-separable group.
Isaacs proved in \cite[Proposition~10.1]{Isaacs86} that the determinant of the Cartan matrix is a $\pi'$-number.
In this paper, we establish $\pi$-analogs of the classical results on the determinant and the elementary divisors.
Recall that two matrices of the same size are called equivalent if one is obtained from the other by multiplication with unimodular matrices on the left and right.
For a finite $\pi$-separable group $G$, let $\cl(G^0)$ denote a set of representatives of the $\pi$-classes.
\index{$cl(G^0)$ @ $\cl(G^0)$ : Set}

\begin{mainthm}
  \label{main:thm}
  Let $\pi$ be a set of primes and let $G$ be a finite $\pi$-separable group.
  Then the Cartan matrix $C$ for the $\pi$-partial characters of $G$ is equivalent over the integers to a matrix $\diag(|\C_G(x)|_{\pi'})_{x \in \cl(G^0)}$.
  In particular,
  \[
    \det C = \prod_{x \in \cl(G^0)} |\C_G(x)|_{\pi'}.
  \]
\end{mainthm}

\begin{maincor}
  Let $\pi$ be a set of primes and let $G$ be a finite $\pi$-separable group.
  If $p \in \pi'$, then the $p$-parts of the elementary divisors of the Cartan matrix for the $\pi$-partial characters of $G$ are $|\C_G(x)|_p$, where $x \in \cl(G^0)$.
\end{maincor}

\index{$p$ : Prime}

\begin{maincor}
  \label{main:largest}
  Let $\pi$ be a set of primes and let $G$ be a finite $\pi$-separable group.
  Then the largest elementary divisor of the Cartan matrix for the $\pi$-partial characters of $G$ is equal to $|G|_{\pi'}$.
\end{maincor}

These corollaries follow immediately from the theorem.
If $p$ is a prime, $\pi = p'$ and $G$ is a finite $p$-solvable group, then the $\pi$-partial characters coincide with the $p$-Brauer characters by the Fong--Swan theorem.
Hence our results recover the classical theorems in this case.
\cref{main:largest} also generalizes a recent result by Iiyori--Kiyota--Sawabe \cite[Theorem~2.8]{IiyoriKiyotaSawabe26} from finite nilpotent groups to finite $\pi$-separable groups.

\section{Preliminaries}
Let $\pi$ be a set of primes.
The complement of $\pi$ in the set of all primes is denoted by $\pi'$.
For a positive integer $n$, the $\pi$-part $n_\pi$ is the largest divisor of $n$ whose prime divisors lie in $\pi$.
A positive integer $n$ is a $\pi$-number if $n = n_{\pi}$.

Let $G$ be a finite group.
For an element $x$ of $G$, the centralizer of $x$ is denoted by $\C_G(x)$ and the conjugacy class of $x$ is denoted by $x^G$.
\index{$C_G(x)$ @ $\C_G(x)$ : Group}
\index{$x^G$ : Set}
An element $x$ of $G$ is a $\pi$-element if its order is a $\pi$-number.
The set of $\pi$-elements is denoted by $G^0$ and $G$ is a $\pi$-group if its order is a $\pi$-number.
\index{$G^0$ : Set}
Every element $x$ of $G$ can be written uniquely as $x = x_\pi x_{\pi'} = x_{\pi'} x_\pi$ with a $\pi$-element $x_\pi$ and a $\pi'$-element $x_{\pi'}$.
A conjugacy class consisting of $\pi$-elements is a $\pi$-class.
Let $\cl(G^0)$ denote a set of representatives of the $\pi$-classes of $G$.
Recall that $G$ is called $\pi$-separable if there is a normal series
\[
  1 = N_0 \leq N_1 \leq \dotsb \leq N_k = G
\]
such that each $N_i/N_{i - 1}$ is a $\pi$-group or a $\pi'$-group for $1 \leq i \leq k$.
\index{$i$ : Integer}
\index{$k$ : Integer}
\index{$N_i$ : Group}
Throughout this paper, we assume that $\pi$ is a set of primes and $G$ is a finite $\pi$-separable group.

\subsection*{Partial characters}
Let $\Irr(G)$ denote the set of irreducible characters of $G$.
\index{$Irr(G)$ @ $\Irr(G)$ : Set}
For a virtual character $\psi$ of $G$, the restriction of $\psi$ to the $\pi$-elements $G^0$ is denoted by $\psi^0$.
\index{$\psi$ : Virtual Ordinary Character}
\index{$\psi^0$ : Partial Character}
Restrictions of characters to the $\pi$-elements are called $\pi$-partial characters.
A $\pi$-partial character is irreducible if it cannot be written as a sum of two $\pi$-partial characters.
The set of irreducible $\pi$-partial characters of $G$ is denoted by $\I_\pi(G)$.
\index{$I_\pi(G)$ @ $\I_\pi(G)$ : Set}
The restriction of an irreducible character $\chi$ of $G$ can be written as a linear combination
\[
  \chi^0 = \sum_{\phi \in \I_\pi(G)} d_{\chi\phi} \phi
\]
\index{$\chi$ : Ordinary Character}
\index{$\phi$ : Partial Character}
with non-negative integer coefficients $d_{\chi\phi}$, called the decomposition numbers.
\index{$d_{\chi\phi}$ : Integer}
Since $G$ is $\pi$-separable, $\I_\pi(G)$ is linearly independent \cite[Theorem~3.3]{Isaacs18} and the decomposition numbers $d_{\chi\phi}$ are uniquely determined.
We call
\[
  D = [d_{\chi\phi}]_{\chi \in \Irr(G), \phi \in \I_\pi(G)}
\]
the decomposition matrix, which has full rank.
\index{$D$ : Matrix}
For $\phi \in \I_\pi(G)$, the projective indecomposable character $\eta_\phi$ is defined by
\[
  \eta_\phi = \sum_{\chi \in \Irr(G)} d_{\chi\phi}\chi.
\]
\index{$\eta_\phi$ : Ordinary Character}
See \cite[Chapter~3]{Isaacs18} for the theory of $\pi$-partial characters.

\begin{lem}[Second orthogonality relation]
  \label{lem:second}
  \[
    \sum_{\phi \in \I_\pi(G)} \phi(x) \overline{\eta_\phi(y)} = \delta_{x^G y^G}|\C_G(x)| \qquad (x, y \in G^0).
  \]
\end{lem}
\index{$y$ : Group Element}
\begin{proof}
  By the second orthogonality relation for irreducible characters,
  \begin{align*}
    \sum_{\phi \in \I_\pi(G)} \phi(x) \overline{\eta_\phi(y)}
    &= \sum_{\phi \in \I_\pi(G)} \phi(x) \left(\sum_{\chi \in \Irr(G)} d_{\chi\phi} \overline{\chi(y)}\right) \\
    &= \sum_{\chi \in \Irr(G)} \left(\sum_{\phi \in \I_\pi(G)} d_{\chi\phi} \phi(x)\right) \overline{\chi(y)} \\
    &= \sum_{\chi \in \Irr(G)} \chi(x) \overline{\chi(y)}
    = \delta_{x^G y^G}|\C_G(x)|.
    \qedhere
  \end{align*}
\end{proof}
\index{$\delta$ : Function}

The second orthogonality relation yields the first orthogonality relation
\begin{equation}
  \label{eq:first}
  \sum_{x \in G^0} \eta_\phi(x)\overline{\mu(x)} = \delta_{\phi\mu}|G|
  \qquad
  (\phi, \mu \in \I_\pi(G))
\end{equation}
as in the proof of \cite[Theorem~2.13]{Navarro98}.
The relation \eqref{eq:first} already appears in \cite[Exercise~3.5]{Isaacs18}.

\subsection*{Virtual characters}
For complex functions $\alpha$ and $\beta$ defined on a subset $S$ of $G$, we write
\[
  \langle \alpha, \beta \rangle_S = \frac{1}{|G|}\sum_{x \in S} \alpha(x)\overline{\beta(x)}.
\]
\index{$\alpha$ : Function}
\index{$\beta$ : Function}
For a $\pi$-partial character $\mu$ of $G$, define
\[
  \tilde\mu(x) = \mu(x_\pi) \quad (x \in G)
  \qquad\text{and}\qquad
  \hat\mu(x) =
  \begin{cases}
    |G|_{\pi'}\mu(x) & (x \in G^0) \\
    0                & (x \not\in G^0).
  \end{cases}
\]
\index{$\mu$ : Partial Character}
\index{$\tilde\mu$ : Virtual Ordinary Character}
\index{$\hat\mu$ : Virtual Ordinary Character}
The following are consequences of Brauer's characterization of virtual characters.

\begin{lem}
  \label{lem:virtual}
  If $\mu$ is a $\pi$-partial character, then $\tilde\mu$ and $\hat\mu$ are virtual characters.
\end{lem}
\begin{proof}
  Let $\mu$ be a $\pi$-partial character of $G$ and let $E$ be an elementary subgroup of $G$.
  Write $E = P \times Q$ where $P$ is a $\pi'$-subgroup and $Q$ is a $\pi$-subgroup of $E$.
  \index{$E$ : Group}
  \index{$P$ : Group}
  \index{$Q$ : Group}
  Then the restriction $\mu_Q$ of $\mu$ to $Q$ is a character of $Q$.
  Let $1_P$ denote the principal character of $P$ and $\rho_P$ the regular character of $P$.
  \index{$1_P$ : Ordinary Character}
  \index{$\rho_P$ : Ordinary Character}
  Then the restrictions of $\tilde\mu$ and $\hat\mu$ to $E$ can be written as
  \[
    \tilde\mu_E = 1_P \times \mu_Q
    \qquad\text{and}\qquad
    \hat\mu_E   = \frac{|G|_{\pi'}}{|P|}\rho_P \times \mu_Q,
  \]
  and both are virtual characters of $E$.
  By Brauer's characterization of virtual characters \cite[Theorem~3.4.2(i)]{NagaoTsushima89}, $\tilde\mu$ and $\hat\mu$ are virtual characters of $G$.
\end{proof}

\begin{lem}
  \label{lem:vanish}
  \[
    \left\{\, \psi \in \bigoplus_{\chi \in \Irr(G)} \ZZ \chi \,\middle|\, \psi(x) = 0 \quad (x \not\in G^0) \,\right\}
    = \bigoplus_{\phi \in \I_\pi(G)} \ZZ \eta_\phi.
  \]
\end{lem}
\index{$\psi$ : Virtual Ordinary Character}
\index{$Z$ @ $\ZZ$ : Integers}
\begin{proof}
  The right-hand side is contained in the left-hand side by \cite[Corollary~3.8]{Isaacs18}.
  Let $\psi$ be a virtual character of $G$ such that $\psi(x) = 0$ for $x \not\in G^0$.
  By \cite[Corollary~3.8]{Isaacs18}, we have $\psi = \sum_{\phi \in \I_\pi(G)} a_\phi \eta_\phi$ for some complex numbers $a_\phi$.
  \index{$a_\phi$ : Integer}
  On the one hand,
  \[
    \langle \psi^0, \phi \rangle_{G^0}
    = \sum_{\mu \in \I_\pi(G)} a_\mu \langle \eta_\mu^0, \phi \rangle_{G^0}
    = a_\phi
  \]
  by the first orthogonality relation \eqref{eq:first}.
  On the other hand, by \cref{lem:virtual}, $\langle \psi^0, \phi \rangle_{G^0} = \langle \psi, \tilde\phi \rangle_G$ is an integer.
  Thus the coefficients $a_\phi$ are integers.
\end{proof}

\section{Proof}
Throughout this section, we use the following notation.
Let $m$ be the exponent of $G$.
\index{$m$ : Integer}
Fix a number field $K$ containing a primitive $m$th root of unity with the ring of integers $\OO$.
\index{$K$ : Number Field}
\index{$O$ @ $\OO$ : Integer Ring}
Set $\ell = |\cl(G^0)|$.
\index{$l$ @ $\ell$ : Integer}
As $G$ is $\pi$-separable, \cite[Theorem~3.3]{Isaacs18} gives $\ell  = |\I_\pi(G)|$.
Define
\begin{align*}
  \Phi   &= [\phi(x)]_{\phi \in \I_\pi(G), x \in \cl(G^0)}, \\
  Y      &= [\eta_\phi(x)]_{\phi \in \I_\pi(G), x \in \cl(G^0)}, \\
  \Delta &= \diag(|\C_G(x)|)_{x \in \cl(G^0)}.
\end{align*}
\index{$\Phi$ : Matrix}
\index{$Y$ : Matrix}
\index{$\Delta$ : Matrix}

The Cartan invariants $c_{\phi\mu}$ are defined by
\[
  c_{\phi\mu} = \langle \eta_\phi, \eta_\mu \rangle_G
  \qquad
  (\phi, \mu \in \I_\pi(G))
\]
and we call $C = [c_{\phi\mu}]_{\phi, \mu \in \I_\pi(G)}$ the Cartan matrix.
\index{$c_{\phi\mu}$ : Integer}
By the first orthogonality relation,
\begin{equation}
  \label{eq:cartan}
  c_{\phi\mu} = \sum_{\chi \in \Irr(G)} d_{\chi\phi}d_{\chi\mu}
\end{equation}
and hence
\[
  \eta_\phi^0
  = \sum_{\chi \in \Irr(G)} d_{\chi\phi} \chi^0
  = \sum_{\mu \in \I_\pi(G)} c_{\phi\mu} \mu.
\]
Thus the Cartan matrix represents a homomorphism
\[
  \bigoplus_{\phi \in \I_\pi(G)} \ZZ \eta_\phi \to \bigoplus_{\mu \in \I_\pi(G)} \ZZ \mu
\]
of abelian groups defined by the restriction.
Let
\[
  A = \bigoplus_{\mu \in \I_\pi(G)} \ZZ \mu \bigg/ \bigoplus_{\phi \in \I_\pi(G)} \ZZ \eta_\phi^0
\]
denote the cokernel of the homomorphism.
\index{$A$ : Group}

\subsection*{Primary decomposition}
By the theory of modules over principal ideal domains, two integer matrices of the same size are equivalent if their cokernels are isomorphic.
We prove that the cokernel $A$ is a torsion group and decompose it into $p$-primary components with $p \in \pi'$.
\begin{lem}
  \label{lem:decomposition}
  \[
    A \cong \bigoplus_{p \in \pi'} A_{(p)}.
  \]
\end{lem}
\begin{proof}
  First we prove that
  \[
    \bigoplus_{\mu \in \I_\pi(G)} |G|_{\pi'} \ZZ \mu \subseteq \bigoplus_{\phi \in \I_\pi(G)} \ZZ \eta_\phi^0.
  \]
  Let $\mu \in \I_\pi(G)$.
  By \cref{lem:virtual,lem:vanish}, $\hat\mu = \sum_{\phi \in \I_\pi(G)} a_\phi \eta_\phi$ for some integers $a_\phi$.
  Restriction to $\pi$-elements yields $|G|_{\pi'}\mu = \sum_{\phi \in \I_\pi(G)} a_\phi \eta_\phi^0$.

  By the Chinese Remainder Theorem \cite[VII, \S2, No.~2, Theorem~1]{Bourbaki90}, $A$ decomposes into its $p$-primary components, which are isomorphic to $A \otimes_\ZZ \ZZ_{(p)} \cong A_{(p)}$.
\end{proof}

\subsection*{Primary components}
We determine the $p$-primary components of the cokernel $A$.
The next lemma is proved using Brauer's theory.
\begin{lem}
  \label{lem:unimodular}
  For a prime $p \in \pi'$ and a prime ideal $\pp$ of $\OO$ lying over $p$,
  \[
    \Phi \in \GL_\ell(\OO_\pp).
  \]
\end{lem}
\index{$p$ @ $\pp$ : Prime Ideal}
\begin{proof}
  Let $\IBr_p(G)$ denote the set of irreducible $p$-Brauer characters of $G$.
  Recall that every $\theta \in \IBr_p(G)$ is an integer linear combination of restrictions of $\chi \in \Irr(G)$ to $p'$-elements by \cite[Corollary~2.16]{Navarro98}.
  \index{$\theta$ : Brauer Character}
  \index{$IBr_p(G)$ @ $\IBr_p(G)$ : Set}
  Hence restriction to $\pi$-elements defines a homomorphism
  \[
    \bigoplus_{\theta \in \IBr_p(G)} \ZZ \theta \to \bigoplus_{\phi \in \I_\pi(G)} \ZZ \phi.
  \]
  Let $\Theta = [\theta(x)]_{\theta \in \IBr_p(G), x \in \cl(G^0)}$.
  \index{$\Theta$ : Matrix}
  Then $\Theta = M\Phi$ for some integer matrix $M$.
  \index{$M$ : Matrix}

  Let $\OO \to \OO/\pp$, $\lambda \mapsto \lambda^*$ denote the canonical homomorphism.
  \index{$\lambda$ : Algebraic Number}
  \index{$*$ : Reduction}
  In what follows, we also use $*$ to denote reduction modulo $\pp$.
  Since $\pi$-elements are $p'$-elements and $\{\, \theta^* \mid \theta \in \IBr_p(G) \,\}$ is linearly independent over the residue field $\OO/\pp$ by \cite[Theorem~1.19]{Navarro98}, we have $\rank \Theta^* = \ell$.
  From
  \[
    \ell = \rank \Theta^* = \rank M^*\Phi^* \leq \rank \Phi^* \leq \ell,
  \]
  it follows that $\Phi^* \in \GL_\ell(\OO/\pp)$ and $\det \Phi \in \OO \setminus \pp$.
  Hence $\Phi \in \GL_\ell(\OO_\pp)$.
\end{proof}

\begin{lem}
  \label{lem:components}
  For a prime $p \in \pi'$,
  \[
    A_{(p)} \cong \bigoplus_{x \in \cl(G^0)} \ZZ/|\C_G(x)|_p\ZZ.
  \]
\end{lem}
\begin{proof}
  Let $\Phi^\dagger$ denote the conjugate transpose of $\Phi$.
  From $Y = C \Phi$ and $\Phi^\dagger Y = \Delta$ by the second orthogonality relation (\cref{lem:second}),
  \[
    \Phi^\dagger C \Phi = \Delta.
  \]
  Complex conjugation permutes $\I_\pi(G)$, and so $\det \Phi^\dagger = \pm \det \Phi$.
  By \cref{lem:unimodular}, $C$ is equivalent over $\OO_\pp$ to $\Delta$.
  
  Let $\delta_k(C)$ denote a greatest common divisor of all the minors of order $k$ of $C$.
  \index{$k$ : Integer}
  \index{$\delta_k(C)$ : Integer}
  It follows from \cite[VII, \S4, No.~6, Corollary~2 and Proposition~6]{Bourbaki90} that
  \[
    \delta_k(C) \OO_\pp = \delta_k(\Delta) \OO_\pp \qquad (1 \leq k \leq \ell).
  \]
  Intersecting both sides with $\QQ$ yields
  \[
    \delta_k(C) \ZZ_{(p)} = \delta_k(\Delta) \ZZ_{(p)} \qquad (1 \leq k \leq \ell)
  \]
  as $n\OO_\pp \cap \QQ = n\ZZ_{(p)}$ for an integer $n$.
  \index{$Q$ @ $\QQ$ : Rationals}
  \index{$n$ : Integer}
  By \cite[VII, \S4, No.~6, Corollary~2 and Proposition~6]{Bourbaki90} again, $C$ is equivalent  over $\ZZ_{(p)}$ to $\Delta$.
  Therefore,
  \[
    A_{(p)} \cong \bigoplus_{x \in \cl(G^0)} \ZZ_{(p)}/|\C_G(x)|\ZZ_{(p)} \cong \bigoplus_{x \in \cl(G^0)} \ZZ/|\C_G(x)|_p\ZZ.
    \qedhere
  \]
\end{proof}

\subsection*{Main theorem}
Combining the primary decomposition (\cref{lem:decomposition}) and primary components (\cref{lem:components}) yields the proof of \cref{main:thm}.
\begin{proof}[Proof of \cref{main:thm}]
  Recall that the Cartan matrix $C$ represents the restriction
  \[
    \bigoplus_{\phi \in \I_\pi(G)} \ZZ \eta_\phi \to \bigoplus_{\mu \in \I_\pi(G)} \ZZ \mu.
  \]
  By \cref{lem:decomposition} and \cref{lem:components},
  \begin{align*}
    A
    \cong \bigoplus_{p \in \pi'}  \bigoplus_{x \in \cl(G^0)} \ZZ/|\C_G(x)|_p\ZZ
    \cong \bigoplus_{x \in \cl(G^0)} \ZZ/|\C_G(x)|_{\pi'}\ZZ.
  \end{align*}
  Hence $C$ is equivalent over the integers to $\diag(|\C_G(x)|_{\pi'})$.

  Also the above shows $|\det C| = \prod |\C_G(x)|_{\pi'}$ by \cite[VII, \S4, No.~7, Corollary~3]{Bourbaki90}.
  Since $C$ is positive definite by \eqref{eq:cartan}, $\det C = \prod |\C_G(x)|_{\pi'}$.
\end{proof}

\begin{rmk}
  It may be worth noting that, from the above proofs,
  \[
    \det \Phi^\dagger \Phi = \prod_{x \in \cl(G^0)} |\C_G(x)|_\pi.
  \]
\end{rmk}

\appendix
\section{Example}
We illustrate our results with an example.
Consider the dihedral group
\[
  G = \langle\, r, s \mid r^{105} = s^2 = 1,\ r^s = r^{-1} \,\rangle
\]
of order $210 = 2\cdot3\cdot5\cdot7$ and $\pi = \{2, 5\}$;
note that $|G|_\pi$ and $|G|_{\pi'}$ are not prime powers.
\index{$r$ : Group Element}
\index{$s$ : Group Element}
The irreducible characters and the irreducible $\pi$-partial characters of $G$ are given in \cref{tbl:Irr(G),tbl:I_pi(G)}.
\begin{table}[ht]
  \centering
  \caption{$\Irr(G)$ ($1 \leq i, j \leq 52$, $\zeta = e^{2\pi\sqrt{-1}/105}$).}
  \label{tbl:Irr(G)}
  \begin{tabular}{cccc}
    \toprule
             & $1^G$ & $s^G$ & $(r^j)^G$ \\
    \midrule
    $\chi_1$ & $1$   & $1$   & $1$ \\
    $\chi_2$ & $1$   & $-1$  & $1$ \\
    $\psi_i$ & $2$   & $0$   & $\zeta^{ij} + \zeta^{-ij}$ \\
    \bottomrule
  \end{tabular}
\end{table}
\index{$i$ : Integer}
\index{$j$ : Integer}
\index{$\zeta$ : Algebraic Number}
\index{$e$ : Transcendental Number}
\index{$\pi$ : Transcendental Number}
\index{$\chi_1, \chi_2$ : Ordinary Character}
\index{$\psi_n$ : Ordinary Character}
\begin{table}[ht]
  \centering
  \caption{$\I_\pi(G)$ ($\tau = 2\cos \frac{\pi}{5} = \frac{1 + \sqrt{5}}{2}$).}
  \label{tbl:I_pi(G)}
  \begin{tabular}{ccccc}
    \toprule
              & $1^G$ & $s^G$ & $(r^{21})^G$ & $(r^{42})^G$\\
    \midrule
    $\phi_1$  & $1$   & $1$   & $1$          & $1$ \\
    $\phi_2$  & $1$   & $-1$  & $1$          & $1$ \\
    $\mu_1$   & $2$   & $0$   & $-1 + \tau$  & $-\tau$ \\
    $\mu_2$   & $2$   & $0$   & $-\tau$      & $-1 + \tau$ \\
    \bottomrule
  \end{tabular}
\end{table}
\index{$\tau$ : Algebraic Number}
\index{$\phi_1, \phi_2$ : Partial Character}
\index{$\mu_1, \mu_2$ : Partial Character}
Then $\chi_1^0 = \phi_1$, $\chi_2^0 = \phi_2$,
\[
  \psi_i^0 =
  \begin{cases}
    \phi_1 + \phi_2 & (i \equiv 0 \mod 5) \\
    \mu_1           & (i \equiv \pm 1 \mod 5) \\
    \mu_2           & (i \equiv \pm 2 \mod 5)
  \end{cases}
\]
and hence
\[
  \left\{
  \begin{alignedat}{2}
      \eta_{\phi_1}^0 &= \chi_1^0 + \sum_{i \equiv 0} \psi_i^0 & &= 11 \phi_1 + 10 \phi_2 \\
      \eta_{\phi_2}^0 &= \chi_2^0 + \sum_{i \equiv 0} \psi_i^0 & &= 10 \phi_1 + 11 \phi_2 \\
      \eta_{\mu_1}^0  &= \sum_{i \equiv \pm 1} \psi_i^0        & &= 21 \mu_1 \\
      \eta_{\mu_2}^0  &= \sum_{i \equiv \pm 2} \psi_i^0        & &= 21 \mu_2.
  \end{alignedat}
  \right.
\]
Thus the relevant matrices are
\[
  C =
  \begin{bmatrix}
    11 & 10 &    &    \\
    10 & 11 &    &    \\
       &    & 21 &    \\
       &    &    & 21
  \end{bmatrix},
  \qquad
  \Phi =
  \begin{bmatrix}
    1 &  1 & 1         & 1 \\
    1 & -1 & 1         & 1 \\
    2 &  0 & -1 + \tau & -\tau \\
    2 &  0 & -\tau     & -1 + \tau
  \end{bmatrix},
\]
\[
  \Delta =
  \begin{bmatrix}
    210 &   &     &   \\
        & 2 &     &   \\
        &   & 105 &   \\
        &   &     & 105
  \end{bmatrix}
\]
and we obtain the desired identities
\[
  \det C = 9261 = 21 \cdot 1 \cdot 21 \cdot 21 = \prod_{x \in \cl(G^0)} |\C_G(x)|_{\pi'},
\]
\[
  \det \Phi^\dagger \Phi = 500 = 10 \cdot 2 \cdot 5 \cdot 5 = \prod_{x \in \cl(G^0)} |\C_G(x)|_\pi.
\]
Note that the entries of the matrix $\Phi$ even lie in a number field $K = \QQ(\sqrt{5})$ with the ring of integers $\OO = \ZZ[\tau]$.
From
\[
  \det \Phi = 2 \cdot (\sqrt{5})^3,
\]
the matrix $\Phi$ is unimodular over $\OO_\pp$ for $\pp = 3\OO$ and $\pp = 7\OO$, but not over $\OO$.

\section*{Acknowledgments}
This collaboration originated at ``Representations of Finite Groups and Ramifications: A Conference in Honor of Michel Brou\'e’s 80th Birthday and His Fundamental Contributions'' (Shenzhen International Center for Mathematics, May 25--29, 2026).
The authors thank the organizers and acknowledge their financial support.

\printbibliography

@article {Brauer41,
    AUTHOR = {Brauer, Richard},
     TITLE = {On the {C}artan invariants of groups of finite order},
   JOURNAL = {Ann. of Math. (2)},
  FJOURNAL = {Annals of Mathematics. Second Series},
    VOLUME = {42},
      YEAR = {1941},
     PAGES = {53--61},
       DOI = {10.2307/1968986},
}

@article {BrauerNesbitt41,
    AUTHOR = {Brauer, R. and Nesbitt, C.},
     TITLE = {On the modular characters of groups},
   JOURNAL = {Ann. of Math. (2)},
  FJOURNAL = {Annals of Mathematics. Second Series},
    VOLUME = {42},
      YEAR = {1941},
     PAGES = {556--590},
       DOI = {10.2307/1968918},
}

@book {Bourbaki90,
    AUTHOR = {Bourbaki, N.},
     TITLE = {Algebra. {II}. {C}hapters 4--7},
    SERIES = {Elements of Mathematics (Berlin)},
 PUBLISHER = {Springer-Verlag, Berlin},
      YEAR = {1990},
}

@article {IiyoriKiyotaSawabe26,
    AUTHOR = {Iiyori, N. and Kiyota, M. and Sawabe, M.},
     TITLE = {Elementary divisors of matrices obtained from the character table of finite groups},
   JOURNAL = {J. Algebraic Combin.},
  FJOURNAL = {Journal of Algebraic Combinatorics. An International Journal},
    VOLUME = {64},
      YEAR = {2026},
       DOI = {10.1007/s10801-026-01576-x},
}

@article {Isaacs84,
    AUTHOR = {Isaacs, I. M.},
     TITLE = {Characters of {$\pi$}-separable groups},
   JOURNAL = {J. Algebra},
  FJOURNAL = {Journal of Algebra},
    VOLUME = {86},
      YEAR = {1984},
     PAGES = {98--128},
       DOI = {10.1016/0021-8693(84)90058-9},
}

@article {Isaacs86,
    AUTHOR = {Isaacs, I. M.},
     TITLE = {Fong characters in {$\pi$}-separable groups},
   JOURNAL = {J. Algebra},
  FJOURNAL = {Journal of Algebra},
    VOLUME = {99},
      YEAR = {1986},
     PAGES = {89--107},
       DOI = {10.1016/0021-8693(86)90056-6},
}

@book {Isaacs18,
    AUTHOR = {Isaacs, I. Martin},
     TITLE = {Characters of solvable groups},
    SERIES = {Graduate Studies in Mathematics},
    VOLUME = {189},
 PUBLISHER = {American Mathematical Society, Providence, RI},
      YEAR = {2018},
}

@book {NagaoTsushima89,
    AUTHOR = {Nagao, Hirosi and Tsushima, Yukio},
     TITLE = {Representations of finite groups},
 PUBLISHER = {Academic Press, Inc., Boston, MA},
      YEAR = {1989},
}

@book {Navarro98,
    AUTHOR = {Navarro, G.},
     TITLE = {Characters and blocks of finite groups},
    SERIES = {London Mathematical Society Lecture Note Series},
    VOLUME = {250},
 PUBLISHER = {Cambridge University Press, Cambridge},
      YEAR = {1998},
}

@book {Serre77,
    AUTHOR = {Serre, Jean-Pierre},
     TITLE = {Linear representations of finite groups},
    SERIES = {Graduate Texts in Mathematics},
    VOLUME = {42},
 PUBLISHER = {Springer-Verlag, New York-Heidelberg},
      YEAR = {1977},
}

\end{document}